\documentclass[10pt]{article}
\usepackage{amsmath,amsthm,amsfonts,amssymb,latexsym,graphicx}
\usepackage[colorlinks,citecolor=blue,pdfpagemode=UseNone,pdfstartview=FitH]{hyperref}

\allowdisplaybreaks[4]

\theoremstyle{plain}
\newtheorem{theorem}{Theorem}
\newtheorem{lemma}{Lemma}
\newtheorem{corollary}{Corollary}
\newtheorem{proposition}{Proposition}

\theoremstyle{definition}
\newtheorem{example}{Example}
\newtheorem{remark}{Remark}

\newcommand{\E}{\mathbb{E}}
\renewcommand{\P}{\mathbb{P}}
\newcommand{\N}{\mathbb{N}}
\newcommand{\R}{\mathbb{R}}

\newcommand{\EEE}{\mathcal{E}}
\newcommand{\FFF}{\mathcal{F}}

\usepackage{dsfont}
\newcommand{\id}{\mathds{1}}

\renewcommand{\d}{\mathrm{d}}

\title{Merging sequential e-values via martingales}

\author{Vladimir Vovk\thanks%
  {Department of Computer Science,
  Royal Holloway, University of London,
  Egham, Surrey, UK.
  E-mail: \href{mailto:v.vovk@rhul.ac.uk}{v.vovk@rhul.ac.uk}.}
\and Ruodu Wang\thanks%
  {Department of Statistics and Actuarial Science,
  University of Waterloo,
  Waterloo, Ontario, Canada.
  E-mail: \href{mailto:wang@uwaterloo.ca}{wang@uwaterloo.ca}.}}

\begin{document}
\maketitle

\begin{abstract}
  We study the problem of merging sequential or independent e-values into one e-value or e-process.
  We describe a class of e-value merging functions via martingales
  and show that it dominates all merging methods for sequential e-values.
  All admissible methods for constructing e-processes can also be obtained in this way. 
  In the case of merging independent e-values, the situation becomes much more complicated,
  and we provide a general class of such merging functions
  based on martingales applied to reordered data.
\end{abstract}

\section{Introduction}

E-values, as an alternative to the standard statistical notion of p-values,
have received increasing attention in recent years;
see, e.g., \cite{S21,VW21,Grunwald/etal:arXiv1906,WR21,RB23}.
E-values are often obtained through nonnegative martingales
or, more generally, e-processes;
see \cite{Shafer/etal:2011,RRLK20}.
The defining feature of e-processes is that they produce valid e-values under optional stopping,
and therefore they are widely applicable in testing with sequential observations.
This feature of e-processes and martingales is fundamental
in the construction of time-uniform confidence sequences;
see \cite{WRB20,HRMS21}.

A standard approach to constructing an e-value or an e-process from sequentially arriving data
is first to compute an e-value from each observation or each batch of observations,
and then to combine the resulting e-values into a single e-value or an e-process.
A simple example is the likelihood ratio process,
which is the product process of many individual e-values (which are likelihood ratios based on each observation).
For more sophisticated examples of e-values and e-processes constructed from individual e-values,
see \cite{Grunwald/etal:arXiv1906}, \cite{HZ22}, and \cite{WWZ22} in different contexts.

This paper is dedicated to a thorough study of combining independent or sequential e-values.
This problem gives rise to substantial mathematical challenges,
and our results lead to practical guidelines for handling multiple e-values in statistical testing.
Our main finding is that there exists a class of explicit methods
which includes all admissible merging methods for sequential e-values;
for independent e-values, the picture is much more complicated.

We define e-variables (whose realizations are e-values)
and discuss their merging problems in Section \ref{sec:2}. 
A merging function which transforms independent e-values into one e-value is called an ie-merging function,
and one which transforms sequential e-values into one e-value is called an se-merging function.
An example illustrating the differences between sequential and independent e-variables
in a simple testing problem is described in Section \ref{sec:example}. 

We start our theoretical treatment from a class of se-merging functions in Section~\ref{sec:sequential}.
We show that the class of martingale merging functions includes all admissible se-merging functions
(Theorem \ref{thm:inverse}).
The class of martingale merging functions is formulated
using the game-theoretic version of martingales
as defined in \cite{Shafer/Vovk:2019}.
Although this version is different from martingales
defined in measure-theoretic probability theory,
the two versions are closely related and reflect the same basic principle.
The notion of a martingale was introduced by Jean Ville \cite{Ville:1939}
as extension (and correction) of von Mises's \cite{Mises:1928} notion of a gambling system.
Kolmogorov \cite{Kolmogorov:1963} came up with another extension of von Mises's notion
(later but independently a similar extension was proposed
by Loveland \cite{Loveland:1966TAMS,Loveland:1966ZML}).
We pay special attention to the product merging function,
which is the simplest se-merging function,
with a weak form of optimality \cite[Proposition 4.2]{VW21}.
We show that when applied to independent precise e-variables,
the product function has the (undesirable) largest variance among a large class of se-merging functions,
making it sub-optimal in some settings.

In Section \ref{sec:anytime} we define e-processes
and connect them to martingale merging functions.
We show that, in a natural sense, for a process obtained from sequential e-values,
anytime validity (the defining property of e-processes) is equivalent
to being generated from a martingale merging function (Theorem \ref{th:2}).

In Section~\ref{sec:independent} we focus on independent  e-values.
We combine Ville's and Kolmogorov's extensions of von Mises's gambling systems
to obtain a class of ie-merging functions (Theorem \ref{th:3})
called generalized martingale merging functions,
allowing for a reordering of the independent e-variables.
This class, although being a natural generalization of martingale merging functions,
turns out to be a strict subset of the set of all admissible ie-merging functions
(Example \ref{ex:ct}).
A full description of the set of all admissible ie-merging functions remains unclear.
With this open question and several others,
we conclude the paper in Section \ref{sec:conclusion}.

\section{E-variables}
\label{sec:2}

A (statistical) hypothesis $H$ is a collection of probability measures
on a measurable space (the underlying sample space).
A \emph{p-variable} $P$ for a hypothesis $H$ is a random variable that satisfies
$Q(P\le \alpha)\le \alpha$ for all $\alpha \in (0,1)$ and all $Q\in H$.
In other words, a p-variable is stochastically larger
than the uniform distribution $\mathrm{U}[0,1]$.
P-variables are often truncated at $1$ and assumed to take values in $[0,1]$.
An \emph{e-variable} $E$ for a hypothesis $H$ is a $[0,\infty]$-valued random variable
satisfying $\E^Q[E]\le1$ for all $Q\in H$.
We will use \emph{e-values} for the realization of e-variables,
and sometimes we use terms such as ``independent e-values'',
which of course means realized values of independent e-variables.
We fix a positive integer $K$ and set $[K]:=\{1,\dots,K\}$.

As shown in \cite{VW21}, for admissibility results on merging functions,
it is harmless to consider the case of a singleton $H=\{\P\}$
for an atomless probability measure $\P$.
We will assume this throughout the paper.
The underlying probability space $(\Omega,\FFF,\P)$ is implicit
and should be clear from the context.
The notation $\E$ is used for the expectation with respect to $\P$.
Without loss of generality we only consider e-variables $E$
taking values in $[0,\infty)$ (since $E<\infty$ a.s.).

We also use standard terminology in probability theory.
A \emph{filtration} in $(\Omega,\FFF,\P)$ is an increasing sequence
$\FFF_1\subseteq\dots\subseteq\FFF_K$ of sub-$\sigma$-algebras of $\FFF$;
we also set $\FFF_0:=\{\emptyset,\Omega\}$.
A process $(X_k)_{k\in \{0,\dots,K\}}$ adapted to $(\FFF_k)_{k\in\{0,\dots,K\}}$
(i.e., with each $X_k$ being $\FFF_k$-measurable)
is a \emph{martingale} on $(\Omega,\FFF,\P)$
if  $\E[X_{k}\mid\FFF_{k-1}] = X_{k-1}$ for all $k\in [K]$,
and the process is a \emph{supermartingale} if $\E[X_{k}\mid\FFF_{k-1}] \le X_{k-1}$,
where $X_1,\dots,X_K$ are assumed integrable;
all relations (such as equalities and inequalities) involving conditional expectations
are understood in the a.s.\ sense.
By a \emph{test martingale} we mean a nonnegative martingale with initial value 1, $X_0=1$.

E-variables $E_1,\dots,E_K$ are \emph{sequential} if $\E[E_{k} \mid E_1,\dots,E_{k-1}]\le1$ for $k\in [K]$.
For a function $F:[0,\infty)^K\to[0,\infty)$ (we assume all functions to be Borel):
\begin{enumerate}
\item
  $F$ is an \emph{e-merging} function
  if, for any e-variables $E_1,\dots,E_K$,
  $F(E_1,\dots,E_K)$ is an e-variable;
\item
  $F$ is an \emph{se-merging} (sequential e-merging) function
  if, for any sequential e-variables $E_1,\dots,E_K$,
  $F(E_1,\dots,E_K)$ is an e-variable;
\item
  $F$ is an \emph{ie-merging} (independent e-merging) function 
  if, for any independent e-variables $E_1,\dots,E_K$,
  $F(E_1,\dots,E_K)$ is an e-variable.
\end{enumerate}
An se-merging (resp.\ ie-merging) function is \emph{admissible}
if it is not dominated by any other se-merging (resp.\ ie-merging) function.

We say that an e-variable  $E$ is \emph{precise} if $\E[E]=1$,
and an se-merging function $F$ is \emph{precise}
if $\E[F(E_1,\dots,E_K)]=1$ for all precise sequential e-variables $E_1,\dots,E_K$. 
This property is satisfied by all examples of merging functions in \cite{VW21}.
Note that for precise sequential e-variables,
it holds that $\E[E_{k} \mid E_1,\dots,E_{k-1}]=1$ for $k\in [K]$.

Since independent e-variables are also sequential e-variables,
the class of ie-merging functions  contains that of se-merging functions.
We will see later in this paper that these two classes are not identical.
An example of se-merging function is the product function
\[
  (e_1,\dots,e_K)\mapsto\prod_{k=1}^K e_k,
\]
which is shown to be optimal in a weak sense
among all ie-merging functions in \cite[Proposition 4.2]{VW21}.
Moreover, if $E_1,\dots,E_K$ are sequential e-variables,
then $(\prod_{k=1}^t E_k)_{t\in\{0,\dots,K\}}$ is a supermartingale.

There is a small difference between our definitions and the ones in \cite{VW21};
namely, we do not require ie-merging and se-merging functions
to be increasing (in the non-strict sense) in all arguments.
This relaxation is quite natural under the betting interpretation
(see Section \ref{sec:sequential}):
if we gain evidence in an early round of betting,
then we may reduce our bet in the next round
(as in, e.g., the ``hit-and-stop'' strategy
in Example \ref{ex:hit-and-stop} below),
which leads to non-monotonicity of the resulting merging function.  

The next result shows that it suffices to consider
bi-valued e-variables in search for ie-merging functions.
\begin{proposition}\label{prop:two-point}
  A function $F:[0,\infty)^K \to [0,\infty)$ is an ie-merging function if and only if
  $F(E_1,\dots,E_K)\le 1$ for all independent e-variables $E_1,\dots,E_K$
  each taking at most two values.
\end{proposition}
\begin{proof}
  The ``only if'' statement is straightforward.
  Below we show the ``if'' statement.
  Let $E_1,\dots,E_K$ be independent e-variables; their means are at most $1$.
  Denote by $\mu_k$ the distribution of $E_k$ for $k\in [K]$.
  Note that any distribution with a finite mean can be written
  as a mixture of bi-atomic distributions with the same mean
  (see \cite[Lemma 2.7]{W14}, which also shows that the bi-atomic distributions
  can be chosen indexed by $(0,1)$).
  Therefore, for each $k\in[K]$, we have a decomposition
  $\mu_k = \int_\R \mu_{k,t} \nu_k(\d t)$,
  where each $\mu_{k,t}$ is a bi-atomic distribution with mean at most $1$,
  and $\nu_k$ is a Borel probability measure on $\R$.
  If $F$ merges all bi-valued independent e-variables into an e-variable,
  then
  \[
    \int_{\R^K} F(e_1,\dots,e_K) \prod_{k=1}^K \mu_{k,t_k} (\d e_k) \le 1
  \]
  for all $t_1,\dots,t_K\in\R$.
  It follows that
  \begin{align*}
    \E[ F(E_1,\dots,E_K)]
    &=
    \int_{\R^K}
    \left(
      \int_{\R^K} F(e_1,\dots,e_K) \prod_{k=1}^K \mu_{k,t_k} (\d e_k)
    \right)
    \prod_{k=1}^K \nu_{k}
    (\d t_k)\\
    &\le
    \int_{\R^K}
    \prod_{k=1}^K \nu_{k}(\d t_k)
    =
    1.
  \end{align*}
  This proves the desired ``if'' statement.
\end{proof}
If $F$ is increasing,
then in the ``if'' statement of Proposition \ref{prop:two-point},
it suffices to consider bi-valued e-variables with mean $1$.

\begin{remark} 
  It may be interesting to note that the situations with merging p-values and e-values
  appear to be opposite.
  In this remark we will concentrate on increasing merging functions.
  Merging independent p-values is in some sense trivial:
  for any measurable increasing function $F:[0,1]^K\to\R$ (intuitively, a test statistic),
  the function
  \[
    G(p_1,\dots,p_K)
    :=
   \lambda(\{(q_1,\dots,q_K)\in[0,1]^K\mid F(q_1,\dots,q_K)\le F(p_1,\dots,p_K)\}),
  \]
  where $\lambda$ is the Lebesgue measure on $[0,1]^K$,
  is an ip-merging function
  (an increasing function transforming independent p-values into a p-value),
  and any ip-merging function can be obtained in this way.\footnote%
    {For merging p-values, it suffices to consider those uniformly distributed on $[0,1]$,
    and for such p-values, being sequential and being independent are equivalent.
    Therefore, we do not discuss sequential p-values.}
  On the other hand, merging arbitrarily dependent p-values is difficult,
  in the sense that the structure of the class of all p-merging functions is very complicated
  (see, e.g., \cite{VWW22}, which includes a review of previous results).
  In the case of e-values,
  merging arbitrarily dependent e-values is trivial,
  at least in the case of symmetric merging functions:
  according to \cite[Proposition 3.1]{VW21},
  arithmetic mean essentially dominates any symmetric e-merging function
  (and \cite[Theorem 3.2]{VW21} gives a full description
  of the class of all symmetric e-merging functions).
  Merging sequential and, especially, independent e-values is difficult
  and is the topic of this paper.
\end{remark}

\section{Sequential and independent e-values}
\label{sec:example}

Before presenting our theoretical results,
we describe a simple example illustrating the difference
between independent and sequential e-variables.

Suppose that a scientist is interested in a parameter $\theta_{\textrm{tr}}\in\Theta$,
and iid observations $X_1,\dots,X_K$ from $\theta_{\textrm{tr}}$
are available and sequentially revealed to her.
She tests $H_0: \theta_{\textrm{tr}}  = 0$ against $H_1: \theta_{\textrm{tr}} \in \Theta_1$
where $0\notin\Theta_1\subseteq\Theta$.
(It does not hurt to think about testing
the Gaussian family $\mathrm{N}(\theta,1)$.)
Let $\ell$ be the likelihood ratio function given by
\begin{equation}\label{eq:LR}
  \ell(x;\theta) = \frac{\d Q_\theta}{\d Q_0}(x),
\end{equation}
where $Q_{\theta}$ is the probability measure for one observation
that corresponds to $\theta\in\Theta$.
It is clear that $\ell (X_k;\theta)$
for any $\theta\in\Theta$ and $k\in [K]$ is an e-variable for $H_0$.  
The scientist may choose one of two strategies
(the second being more general):
\begin{enumerate} 
\item[(a)]
  Fix $\theta_1,\dots,\theta_K\in \Theta_1$,
  and define the e-variables $E_k:= \ell (X_k;\theta_k)$ for $k\in [K]$.
  One may simply choose all $\theta_k$ to be the same.
\item[(b)]
  Choose $\theta_1,\dots,\theta_K$ adaptively,
  where $\theta_k$ is estimated from $(X_1,\dots,X_{k-1})$ for each $k$.
  This can be obtained, e.g., as a point estimate of $\theta$
  or by a Bayesian update rule for some prior on $\Theta_1$
  (for a Bayesian update, the mixture may well be outside the model,
  in which case the numerator of \eqref{eq:LR} would have to be replaced
  by a mixture of $Q_{\theta}$ over $\theta\in\Theta_1$).
  Define the e-variables $E_k:=\ell(X_k;\theta_k)$ for $k\in[K]$.
\end{enumerate}
The scientist then combines the e-variables $E_k$, $k\in[K]$,
to form a final e-variable $\prod_{k=1}^K E_k$
(e.g., for using e-values as weights in multiple testing; see \cite{IWR22}).
Both methods produce a valid final e-variable.
Indeed, we can verify that, in (a), $E_1,\dots,E_K$ are independent e-variables,
and in (b), $E_1,\dots,E_K$ are sequential e-variables.

We illustrate the two supermartingales obtained from (a) and (b) in a simple example.
Suppose that an iid sample  $(X_1,\dots,X_K)$ from $\mathrm{N}(\theta_{\textrm{tr}},1)$ is available.
The null hypothesis is $\theta_{\textrm{tr}} = 0$, and the alternative is $\theta_{\textrm{tr}}>0$. 
We set $\theta_{\textrm{tr}}:=0.3$.
We consider five different ways of constructing
$E_k := \ell(X_k;\theta_k)$ for $k\in [K]$:
(i) choose $\theta_k := \theta_{\textrm{tr}}=0.3$ (knowing the true alternative),
  which is growth-optimal (see, e.g., \cite{S21,Grunwald/etal:arXiv1906});
(ii) choose $\theta_k := 0.1$, which is a misspecified alternative;
(iii) choose iid $\theta_k$ following the uniform distribution on $[0,0.5]$;
(iv) choose $\theta_k$ by the Bayesian update with prior $\theta\sim\mathrm{N}(0.1,0.2^2)$;
(v) choose $\theta_k$ with $\theta_1:=0.1$ and $\theta_k$ the maximum likelihood estimate of $\theta$
  based on $X_1,\dots,X_{k-1}$.
The resulting supermartingales (on the logarithmic scale) are plotted in Figure~\ref{fig:1} with $K=500$.
We note that methods (i), (ii), and (iii) are based
on combining independent e-values with the product merging function,
and (iv) and (v) are based on combining sequential e-values with the same merging function.
For sufficiently large sample sizes, the methods (iv) and (v) based on sequential e-values
are more powerful than (ii) and (iii) based on misspecified or random alternatives,
because the latter are not adaptive.
\begin{figure}[t]
  \begin{center}
    \includegraphics[height=5.5cm, trim={0 15 0 5}, clip]{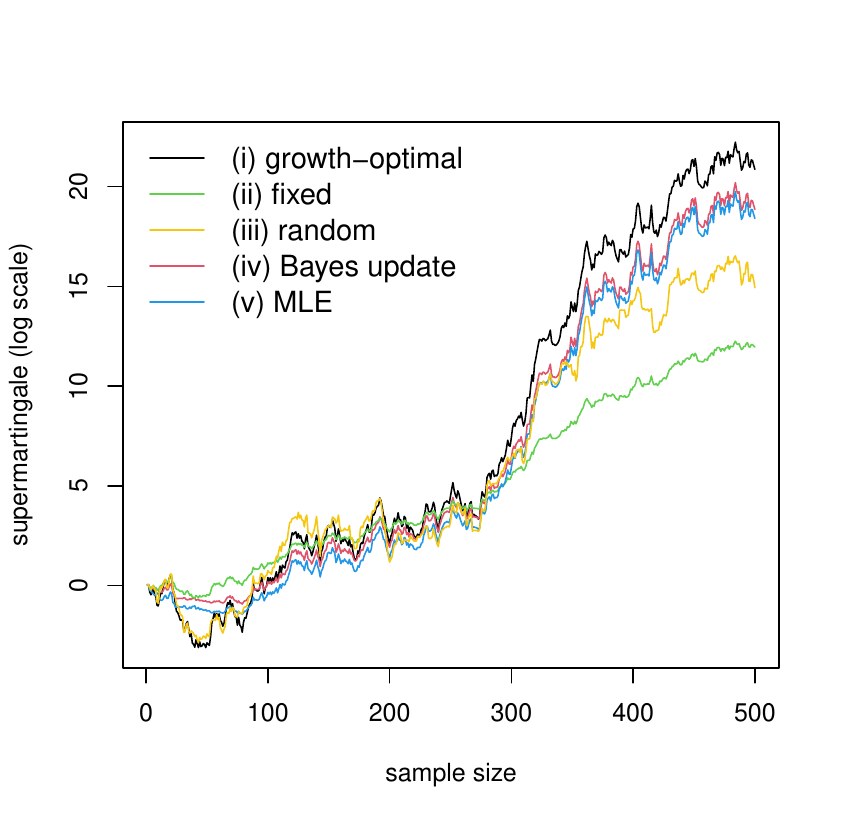}
    \includegraphics[height=5.5cm, trim={0 15 0 5}, clip]{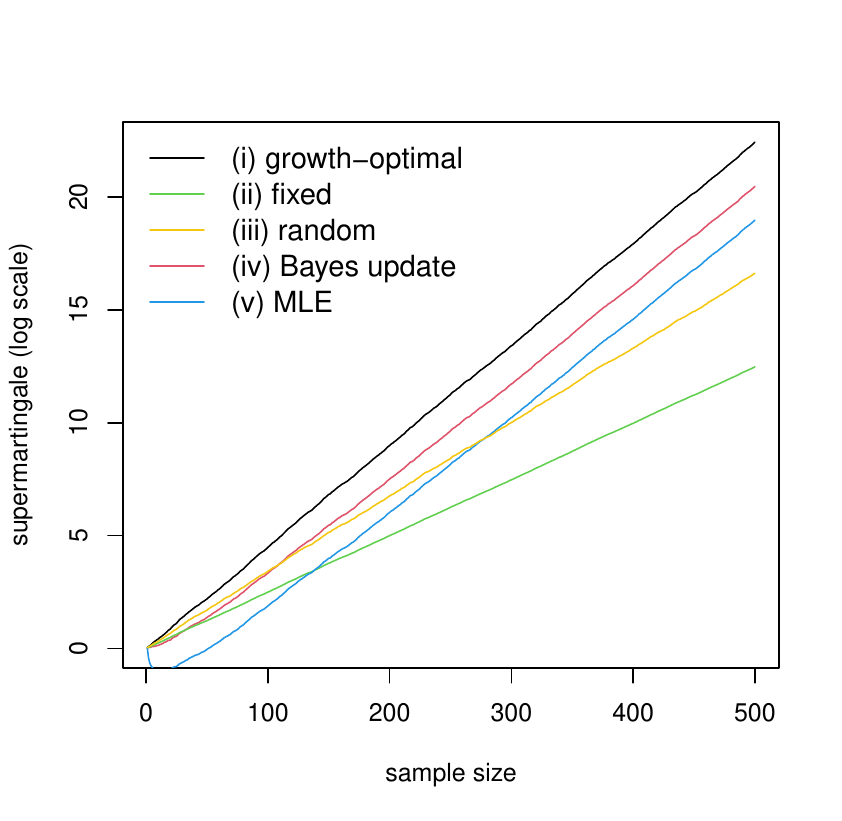}
  \end{center}
  \caption{A few ways of constructing supermartingales from likelihood ratio.
    Left: one run.
    Right: the average of 1000 runs (with average taken on the log values).}
\label{fig:1}
\end{figure}

From this example, we can see that sequential e-variables could be more powerful
(at least in this simple setting),
whereas independent e-variables are more restrictive,
but they allow for more merging methods,
as we will see later in Section \ref{sec:independent}.

\section{Merging sequential e-values}
\label{sec:sequential}

We fix some notation throughout the theoretical development of the paper. 
If $A$ is a measurable space,
we let $A^{<K}$ stand for the measurable space $\bigcup_{k=0}^{K-1}A^k$,
where $A^0:=\{\Box\}$, with $\Box$ denoting the empty sequence. 
For $\mathbf{e}=(e_1,\dots,e_K)\in[0,\infty)^K$ and $k\in[K]$,
we use $\mathbf{e}_{(k)}:=(e_1,\dots,e_k)$
to represent the vector of the first $k$ components of $\mathbf{e}$,
with $\mathbf{e}_{(0)}:=\Box$.
Let $\EEE$ be the class of all e-variables,
i.e., nonnegative random variables $E$ on the underlying probability space
satisfying $\E[E]\le1$.

We first define the notions of a gambling system and a game martingale.
They will form a basis for se-merging functions.
A \emph{gambling system} is a measurable function $s:[0,\infty)^{<K}\to[0,1]$.
The \emph{game martingale}
associated with the gambling system $s$ and initial capital $c\in[0,1]$
is the sequence of functions $S_k:[0,\infty)^K\to[0,\infty)$,
$k=0,\dots,K$,
which is defined recursively by $S_0:=c$ and
\begin{align}\label{eq:S}
  S_{k+1}(\mathbf{e})
  :=
  S_k(\mathbf{e})  
  \bigl(
    s(\mathbf{e}_{(k)}) e_{k+1}
    +
    1-s(\mathbf{e}_{(k)})
  \bigr), 
  \quad
  k=0,\dots,K-1.
\end{align}
(This is a martingale in the generalized sense of \cite{Shafer/Vovk:2019},
which is slightly different from the same notion in measure-theoretic probability theory.)
The intuition is that we observe $e_1,\dots,e_K$ sequentially,
start with capital at most 1,
and at the end of step $k$
invest a fraction $s(\mathbf{e}_{(k)})$ of our current capital in $e_{k+1}$,
leaving the remaining capital aside.
We will also say that we \emph{gamble} the fraction $s$ of our capital
and refer to $s$ as our \emph{bet}.
Then $S_k(\mathbf{e})$,
which depends on $\mathbf{e}$ only via $\mathbf{e}_{(k)}$,
is our resulting capital at time $k$.

\begin{lemma}\label{lem:combination}
  A convex combination of game martingales is a game martingale.
\end{lemma}

\begin{proof}
  The statement of the lemma follows from the following equivalent (and often useful) definition:
  a game martingale is a sequence of nonnegative functions $S_k:[0,\infty)^K\to[0,\infty)$,
  $k=0,\dots,K$,
  such that $S_0\le1$ and, for some measurable function $t:[0,\infty)^{<K}\to[0,\infty)$,
  we have
  \begin{equation}\label{eq:equivalent}
    S_{k+1}(\mathbf{e})
    =
    S_k(\mathbf{e})
    +
    t(\mathbf{e}_{(k)})
    (e_{k+1} - 1)
  \end{equation}
  for all $k=0,\dots,K-1$ and all $\mathbf{e}\in[0,\infty)^K$.
\end{proof}

A \emph{martingale merging function} is a function $F:[0,\infty)^K\to[0,\infty)$
that can be represented in the form $F=S_K$
for some game martingale $S_k$, $k=0,\dots,K$.
An equivalent formulation is 
\begin{equation}\label{eq:equivalent-2}
  S_{K}(\mathbf{e})
  =
  S_0
  \prod_{k=1}^K
  (1 + s(\mathbf{e}_{(k-1)})
  (e_{k} - 1)),
\end{equation}
where $s$ is a gambling system.
The following two results show that martingale merging functions are se-merging functions
and that any se-merging function is dominated by a martingale merging function.

\begin{lemma}\label{lem:mart_is_se}
  Any martingale merging function is an se-merging function.
\end{lemma}

\begin{proof}
  Let $E_1,\dots,E_K$ be sequential e-variables.
  Then $(S_k(E_1,\dots,E_K),\FFF_k)$, $k=0,\dots,K$,
  where $S_k$ is defined by \eqref{eq:S}
  and $\FFF_k$ is the $\sigma$-algebra generated by $E_1,\dots,E_k$,
  is a supermartingale.
  This immediately implies $\E[S_K(E_1,\dots,E_K)]\le 1$.
\end{proof}

\begin{theorem}\label{thm:inverse}
  Any se-merging function is dominated by a martingale merging function.
\end{theorem}

\begin{proof}
  Let $F$ be an se-merging function;
  our goal is to construct a dominating martingale merging function.
  First we consider e-variables taking values in the set $2^{-n}\N$,
  where $\N:=\{0,1,\dots\}$;
  let $\EEE_n$ be the set of such e-variables.
  Extend $F$ to shorter sequences of e-values by
  \begin{align}
    F_{n,K}(e_1,\dots,e_K)
    &:=
    F(e_1,\dots,e_K),\notag\\
    F_{n,k}(e_1,\dots,e_k)
    &:=
    \sup_{E\in\EEE_n}
    \E[F_{n,k+1}(e_1,\dots,e_k,E)]
    \label{eq:step}
  \end{align}
  for all $e_1,\dots,e_K\in2^{-n}\N$ and $k=K-1,\dots,0$.
  It is clear that $F_0\le1$.
  By the duality theorem of linear programming,
  for any $k\in\{0,\dots,K-1\}$ and any $e_1,\dots,e_K\in2^{-n}\N$,
  there exists $s\in[0,1]$ such that
  \begin{equation}\label{eq:Sigma}
    \forall e \in 2^{-n}\N:
    F_{n,k+1}(e_1,\dots,e_k,e)
    \le
    (s e + 1 - s) F_{n,k}(e_1,\dots,e_k).
  \end{equation}

  Let us check carefully the application of the duality theorem. 
  Let $c_1,\dots,c_N$ be the first $N$ elements of the set $2^{-n}\N$
  (namely, $c_i:=(i-1)2^{-n}$, $i\in[N]$);
  we are interested in the case $N\to\infty$.
  Restricting $E$ in \eqref{eq:step} to take values $c_1,\dots,c_N$
  with any probabilities $p_1,\dots,p_N$,
  instead of $F_{n,k}(e_1,\dots,e_k)$ we will obtain the solution $F_{n,k,N}$ to the linear programming problem
  \begin{align}
    c_1 p_1 + \dots + c_N p_N &\le 1 \label{eq:LP-1} \\
    p_1 + \dots + p_N &= 1 \\
    f_1 p_1 + \dots + f_N p_N &\to\max \label{eq:LP-2},
  \end{align}
  where $p_i$ are nonnegative variables and $f_i:=F_{n,k+1}(e_1,\dots,e_k,c_i)$, $i=1,\dots,N$.
  It is clear that the sequence $F_{n,k,N}$ is increasing in $N$
  and tends to $F_{n,k}(e_1,\dots,e_k)$ as $N\to\infty$.
  Let us assume $F_{n,k,N}>0$
  (the simple case where $F_{n,k,N}>0$ for all $N$
  should be considered separately).
  The dual problem to \eqref{eq:LP-1}--\eqref{eq:LP-2} is
  $y_1+y_2\to\min$ subject to $y_1\ge0$ and $c_i y_1 + y_2 \ge f_i$ for all $i\in[N]$.
  Then we will have the analogue
  \begin{equation}\label{eq:pre-Sigma}
    \forall e \in 2^{-n}\{0,\dots,N-1\}:
    F_{n,k+1}(e_1,\dots,e_k,e)
    \le 
    (s e + 1 - s)
    F_{n,k,N}
  \end{equation}
  of \eqref{eq:Sigma} when $y_1+y_2=F_{n,k,N}$
  (which is the case for the optimal $(y_1,y_2)$)
  and $y_1=s F_{n,k,N}$ (which is the definition of $s$).
  It is clear from \eqref{eq:pre-Sigma} that $s\le1$, as $e=0$ is allowed.
  Let $s_N$ be an $s$ satisfying \eqref{eq:pre-Sigma}.
  Then any limit point of the sequence $s_N$ will satisfy \eqref{eq:Sigma};
  such limit points exist since $s_N$ lies in the compact set $[0,1]$.

  We have proved the statement of the theorem for e-values in $2^{-n}\N$;
  now we drop this assumption.
  Let $(e_1,\dots,e_K)\in[0,\infty)^K$.
  For each $n$, let $e_{n,k}$ be the largest number in $2^{-n}\N$ that does not exceed $e_k$.
  Set
  \begin{equation*}
    F_k(e_1,\dots,e_k)
    :=
    \lim_{n\to\infty}
    F_{n,k}(e_{n,1},\dots,e_{n,k}).
  \end{equation*}
  Then $F_k$ is a game martingale,
  and the fraction $s$ to gamble after observing $e_1,\dots,e_k$
  can be chosen as the smallest $s\in[0,1]$ satisfying
  \begin{equation}\label{eq:super-Sigma}
    \forall e \in [0,\infty):
    F_{k+1}(e_1,\dots,e_k,e)
    \le
    F_{k}(e_1,\dots,e_k)
    (s e + 1 - s).
  \end{equation}
  The set of such $s$ is obviously closed; let us check that it is non-empty.
  Let $s=s_n$ be a number in $[0,1]$ satisfying \eqref{eq:Sigma}
  with $e_{n,1},\dots,e_{n,k}$ in place of $e_1,\dots,e_k$, respectively.
  Then any limit point of $s_n$ will satisfy \eqref{eq:super-Sigma}. 
\end{proof}

Theorem~\ref{thm:inverse} gives a characterization
of {admissible} se-merging functions.

\begin{corollary}
  The class of all admissible se-merging functions coincides
  with the class of all martingale merging functions.
\end{corollary}

As we mentioned in Section \ref{sec:2},
a martingale merging function is not necessarily increasing in all arguments.
This is because $s(\mathbf{e}_{k-1})(e_k-1)$ in \eqref{eq:equivalent-2}
is generally not increasing or decreasing in $\mathbf{e}_{(k)}$.
Although monotonicity does not hold,
any martingale merging function $F$ satisfies a property of \emph{sequential monotonicity}:
for fixed $k\in [K]$ and  $(e_1,\dots,e_{k-1})\in [0,\infty)^{k-1}$,
the function $e_{k}\mapsto F(e_1,\dots,e_{k-1}, e_k,1,\dots,1)$ is increasing.

\begin{example}\label{ex:hit-and-stop}
  The non-monotonicity of $S_K$ appears naturally in a ``hit-and-stop'' gambling system:
  for a fixed $\alpha\in (0,1)$ and for each $k\in[K]$,
  if $S_k (\mathbf{e}) \ge 1/\alpha$, then we choose $s(\mathbf{e}_{(k)})=0$
  (which implies $s(\mathbf{e}_{(j)})=0$ for all $j\ge k$);
  otherwise we choose $s(\mathbf{e}_{(k)})>0$. 
  It is clear from \eqref{eq:equivalent-2} that $S_K$ is not increasing
  since $s(\mathbf{e}_{(k-1)})(e_k-1)$ is unbounded from above if $s(\mathbf{e}_{(k-1)})>0$.
\end{example}

\subsection*{Examples of martingale merging functions}

The simplest non-trivial gambling system is $s:=1$;
the corresponding game martingale with initial capital 1 is the product
\[
  S_k(e_1,\dots,e_K)
  =
  e_1\dots e_k,
  \quad
  k\in\{0,\dots,K\},
\]
and the corresponding martingale merging function is the product
\begin{align}\label{eq:product}
  P_K(e_1,\dots,e_K)
  :=
  e_1\dots e_K.
\end{align}
This is the most standard se-merging function.

Another martingale merging function is the arithmetic mean
\[
  F(e_1,\dots,e_K)
  :=
  \frac{e_1+\dots+e_K}{K}.
\]
This is in fact an e-merging function
(the most important symmetric one, as explained in \cite{VW21}).
The corresponding game martingale is the mean
\[
  S_k(e_1,\dots,e_K)
  :=
  \frac{e_1+\dots+e_k+K-k}{K}.
\]
(This is easiest to see using the equivalent definition \eqref{eq:equivalent}.)

A more general class of martingale merging functions,
introduced in \cite{VW21},
includes the U-statistics
\begin{equation}\label{eq:U}
  U_n(e_1,\dots,e_K)
  :=
  \frac{1}{\binom{K}{n}}
  \sum_{A \subseteq [K],\lvert A\rvert=n}
  \left(
    \prod_{k\in A} e_{k}
  \right),
  \quad
  n\in\{0,1,\dots,K\}.
\end{equation}
For each $n$,
this is a martingale merging function because each addend in \eqref{eq:U} is,
and a convex combination of game martingales is a game martingale
(Lemma~\ref{lem:combination}).

Our final martingale merging function
has an increasing sequence of numbers $1\le K_1<\dots<K_m<K$ as its parameter
and is defined as
\[
  F(e_1,\dots,e_K)
  :=
  \prod_{i=0}^m
  \frac{e_{K_i+1}+\dots+e_{K_{i+1}}}{K_{i+1}-K_i},
\]
where $K_0$ is understood to be $0$ and $K_{m+1}$ is understood to be $K$.
The corresponding game martingale is
\[
  S_k(e_1,\dots,e_K)
  :=
  \frac{\sum_{j=1}^{K_1}e_j}{K_1}
  \dots
  \frac{\sum_{j=K_{i-1}+1}^{K_i} e_j}{K_i-K_{i-1}} \;
  \frac{\sum_{j=K_{i}+1}^{k} e_j+K_{i+1}-K_i-k}{K_{i+1}-K_i},
\]
where $i$ is the largest number such that $K_i\le k$.

In practical applications,
a good choice of the se-merging function
depends on the joint distribution of the sequential e-variables $E_1,\dots,E_K$. 
A common criterion is to make the expected logarithm, also known as the e-power, to be as large as possible,
that is, to maximize $\E[\log F(E_1,\dots,E_K)]$ where the expectation  is computed under a specific alternative;
see, e.g., \cite{Grunwald/etal:arXiv1906}.
Certainly, such a maximizer  depends  crucially on the alternative, which is often not known to the decision maker.
Without an accurate information on the alternative,
a one-parameter family of martingale merging functions indexed by $\lambda \in (0,1]$,
\[
  F_\lambda (e_1,\dots,e_K)
  :=\prod_{k=1}^K (1-\lambda + \lambda e_k),
\]
can be employed, which corresponds to a constant gambling system $s=\lambda$.
Within this family, $\lambda$ can be chosen larger (close to $1$)
if the decision maker has prior knowledge or confidence that $\E[\log E_k]$, $k=1,\dots,K$, are large,
and  $\lambda$ can be chosen smaller (close to $0$)
if the decision maker anticipates that $\E[\log E_k]$, $k=1,\dots,K$, may be small.

\subsection*{The product se-merging function}

We pay special attention to the product function $P_K$ in \eqref{eq:product}.
This function is optimal in a weak sense \cite[Proposition 4.2]{VW21} among all ie-merging functions.
In what follows, we show that, when  applied to independent precise e-variables,
the product function has the largest variance among precise se-merging functions.
This is an undesirable property of the product function,
and it is one of the motivations that led us to search for alternative se-merging functions.

We first present a simple lemma which is useful in the proof of the next result.
\begin{lemma}\label{lem:1}
  For any ie-merging function $F:[0,\infty)^K\to [0,\infty)$ and $a\ge 1$, 
  the function $G:(e_1,\dots,e_K )\mapsto F(a e_1,e_2,\dots,e_K)/a$ 
  is an ie-merging function.
\end{lemma}
 
\begin{proof}
  Take an event $A$ with $\P(A)=1/a$.
  For any independent e-variables $E_1',\dots,E_K'$,
  we can take independent e-variables $E_1,\dots,E_K$ independent of $A$  
  such that $(E_1,\dots,E_K)$ is distributed identically to $(E'_1,\dots,E_K')$,
  since the probability space is atomless.
  Therefore, it suffices to show $\E[G(E_1,\dots,E_K)]\le1$
  for independent e-variables $E_1,\dots,E_K$ independent of $A$.
  Let $E_1' := a E_1\id_A$; this is an e-variable.
  We can compute
  \begin{equation*}
    \E[G(E_1,\dots,E_K)]
    =
    \frac 1 a \E[F(a E_1,E_2,\dots,E_K)]
    \le
    \E[F(E_1',E_2,\dots,E_K)]
    \le
    1
  \end{equation*}
  (the first inequality following from $F\ge0$).
  Hence, $G$ is an ie-merging function.
\end{proof}

Now we can compare the product function and other se-merging functions
with respect to the variance and second moment of the resulting e-variable.

\begin{proposition}\label{prop:2}
  For any se-merging function $F:[0,\infty)^K\to [0,\infty)$
  and independent nonnegative random variables $X_1,\dots,X_K$ with $\E[X_k]\ge 1$, $k\in[K]$,
  we have
  \begin{equation}\label{eq:2nd}
    \E \left[F(X_1,\dots,X_K)^2 \right]
    \le
    \prod_{k=1}^K \E [X_k^2]
    =
    \E\left[P_K(X_1,\dots,X_K)^2\right].
  \end{equation}
  Moreover, if $E_1,\dots,E_K$ are independent precise e-variables 
  and $F$ is a precise se-merging function, then
  \begin{equation}\label{eq:3rd}
    \mathrm{Var}(F(E_1,\dots,E_K ))
    \le
    \mathrm{Var}(P_K(E_1,\dots,E_K)).
  \end{equation}
\end{proposition}

\begin{proof}  
  First, we argue that it suffices to show \eqref{eq:2nd}
  for all se-merging functions $F$ 
  and $X_1,\dots,X_K$ being precise e-variables.
  Suppose that this condition holds true. 
  For independent nonnegative $X_1,\dots,X_K$ with mean larger than or equal to $1$,
  set $a_k:=\E[X_k]\ge1$, $E_k:=X_k/a_k$ for $k\in[K]$,
  and $a:=\prod_{k=1}^K a_k$.
  Let $G:(e_1,\dots,e_K )\mapsto F(a_1 e_1, \dots,a_K e_K)/a$.
  Clearly, $E_1,\dots,E_K$ are independent precise e-variables. 
  Using Lemma \ref{lem:1} repeatedly,
  we deduce that $G$ is an se-merging function.
  If \eqref{eq:2nd} holds for all precise e-variables,
  then 
  \begin{align*}
    \E\left[F(X_1,\dots,X_K)^2\right]
    &=
    a^2
    \E\left[\left(\frac{F(a_1E_1,\dots,a_KE_K)}{a}\right)^2\right] \\
    &=
    a^2
    \E\left[G(E_1,\dots,E_K)^2\right] \\
    &\le
    a^2
    \E\left[P_K(E_1,\dots,E_K)^2\right] \\
    &=
    \E\left[P_K(X_1,\dots,X_K)^2\right].
  \end{align*}
  Therefore, the general case of \eqref{eq:2nd}
  follows from the case of precise e-variables.

  Let $E_1,\dots,E_K$ be independent precise e-variables;
  we will show
  \begin{equation}\label{eq:4th}
    \E\left[F(E_1,\dots,E_K)^2 \right]
    \le
    \prod_{k=1}^K \E[E_k^2].
  \end{equation}
  Using Theorem \ref{thm:inverse}, $F$ is dominated by a martingale merging function.
  Hence, it suffices to show the proposition for a martingale merging function $F$.
  Denote the gambling system of $F$ by $s$ and the associated game martingale by $S$.
  We show the proposition by induction.
  The inequality \eqref{eq:4th} holds for $K=1$ since
  \[
    \mathrm{Var}(1-\lambda +\lambda E_1)
    =
    \lambda^2 \mathrm{Var}(E_1)
    \le
    \mathrm{Var}(E_1)
  \]
  for all $\lambda\in[0,1]$.
  To argue by induction, suppose that
  \begin{equation}\label{eq:5th}
    \E[G(E_1,\dots,E_{K-1})^2] \le \prod_{k=1}^{K-1}\E[E_k^2]
  \end{equation}
  for every se-merging function $G:[0,\infty)^{K-1}\to [0,\infty)$.
  Let us write $\mathbf Y:= (E_1,\dots,E_{K-1})$, $W:=S_{K-1}(\mathbf Y)$,
  and $s :=s(\mathbf Y)$. Since  $ \E[  E_K ^2  ] \ge (\E[  E_K   ])^2= 1,$ 
  we have
  \begin{align*}
    \E[F(E_1,\dots,E_{K})^2 \mid \mathbf Y]
    &=
    \E[W^2(s E_K + 1-s)^2 \mid \mathbf Y]  
    \\
    &= W^2\left(s^2 \E[E_K^2]+ (1-s)^2 +2s(1-s)\E[E_K ]\right)
    \\
    &\le W^2\left(s^2 + (1-s)^2 + 2s(1-s)\right)\E[E_K^2]
    \\
    &= W^2\E[E_K^2].
  \end{align*}
  As a consequence,
  \[
    \E[F(E_1,\dots,E_{K})^2]
    \le
    \E[S_{K-1}(\mathbf Y)^2] \E[E_K^2]
    \le
    \prod_{k=1}^K \E[E_k^2],
  \]
  where the last inequality follows by the inductive assumption \eqref{eq:5th}
  and the fact that $S_{K-1}$ is an se-merging function.
  Therefore, we obtain \eqref{eq:2nd}.
  If $F$ is a precise se-merging function,
  we obtain \eqref{eq:3rd} from \eqref{eq:4th}
  since $\E[F(E_1,\dots,E_K)]=\E[P_K(E_1,\dots,E_K)]=1$.
\end{proof}
 
Proposition \ref{prop:2} says that using the product function $P_K$
results in the largest variance under a null hypothesis
under which the e-variables are precise and independent.
This does not imply that the same holds under alternative hypotheses,
which can be arbitrary.
Nevertheless,  we see from \eqref{eq:2nd} in Proposition~\ref{prop:2} that,
if the e-variables have means larger than $1$
and are independent under the alternative hypothesis,
then the second moment of $P_K(E_1,\dots,E_K)$
is larger than (or equal to) that of $F(E_1,\dots,E_K)$ for any se-merging function $F$.
This suggests that one could anticipate
that $P_K(E_1,\dots,E_K)$ likely has a large variance also in this case.

\begin{remark}
  Proposition \ref{prop:2} says that,
  under some conditions, the product function has the largest second moment,
  but this observation does not generalize to other quantities
  such as the expected logarithm (also known as the e-power). 
  Consider two e-variables $1$ and $E$ where $\E[E]>1$ under the alternative hypothesis,
  and the class of se-merging functions
  $f_{\lambda}:(e_1,e_2)\mapsto e_1(1-\lambda+\lambda e_2)$ indexed by $\lambda \in [0,1]$.
  The expected logarithm of $ (1-\lambda+\lambda E)$ is given by $\E[\log (1-\lambda+\lambda E)]$,
  which is concave in $\lambda$.
  This quantity is not necessarily maximized by $\lambda=1$.
  Indeed, when $\E[\log E]<0$, the maximizer $\lambda$ is strictly between $0$ and $1$.
  Therefore, the product function, corresponding to $\lambda=1$,
  does not necessarily have the largest  expected logarithm.
\end{remark}

\section{E-processes and anytime validity}\label{sec:anytime}

In this section, we connect the martingale merging functions to e-processes. 
All supermartingales, martingales, and stopping times are defined
with respect to the filtration $(\FFF_k)$
(the \emph{natural filtration})
generated by $\mathbf{E}=(E_1,\dots,E_K)$.
Formally, a random variable $\tau$ taking values in $\{0,1,\dots,K\}$
is a \emph{stopping time} if, for each $k\in\{0,1,\dots,K\}$,
the set $\{\tau\le k\}$ is $\FFF_k$-measurable.

In scientific discovery, experiments are often conducted sequentially in time,
and a discovery may be  reported at the time when enough evidence is gathered.
Therefore, with a vector $\mathbf{E}$ of sequential e-values,
it is desirable to require validity of a test not only at the fixed time $K$,
but also at a stopping time $\tau$.
This leads to the definition of e-processes:
an \emph{e-process} $(E_k)_{k\in\{0,\dots,K\}}$ is a nonnegative stochastic process
adapted to $(\FFF_k)$ such that $\E[E_\tau]\le1$
for any stopping time $\tau$.
The property $\E[E_\tau]\le1$ will be called \emph{anytime validity}.
Anytime validity is automatically achieved by using a game martingale:
since $(S_k(\mathbf{E}))_{k=0,\dots,K}$ is then a martingale,
$S_\tau(\mathbf{E})$ is an e-variable for any stopping time $\tau$.
Therefore, such a process $k\mapsto S_k(\mathbf{E})$ is an e-process.

Conversely, if a sequence of functions $F_k:[0,\infty)^K\to[0,\infty)$,
$k=0,\dots,K$,
satisfies, for any sequence of e-values $\mathbf{E}$,
\begin{enumerate}
\item[(a)]
  $(F_k(\mathbf{E}))_{k=0,\dots,K}$ is adapted to the natural filtration of $\mathbf{E}$,
\item[(b)]
  \emph{anytime validity}, i.e., $F_\tau(\mathbf{E})$ is an e-variable
  for any stopping time $\tau$,
\item[(c)]
  \emph{precision}, i.e., $F_k(\mathbf{E})$ is precise for all $k\in\{0,\dots,K\}$,
\end{enumerate}
then we can show that it is a game martingale.

\begin{theorem}\label{th:2}
  For a sequence of functions $F=(F_k)_{k=0,\dots,K}$,
  $F_k:[0,\infty)^K\to[0,\infty)$,
  the following are equivalent: 
  \begin{enumerate}
  \item[(i)]
    $F$ is a game martingale with initial value $1$;
  \item[(ii)]
    $F(\mathbf{E})$ is a martingale with initial value $1$
    for any vector $\mathbf{E}$ of precise and sequential e-variables;
  \item[(iii)]
    $F$ is adapted, anytime valid, and precise,
    i.e., satisfies (a)--(c).
  \end{enumerate}
\end{theorem}

\begin{proof} 
  The implications $\text{(i)}\Rightarrow\text{(ii)}$ and $\text{(ii)}\Rightarrow\text{(iii)}$
  are straightforward.
  Below we show $\text{(iii)}\Rightarrow\text{(i)}$.

  Take any precise and sequential e-variables $E_1,\dots,E_K$,
  and let $\mathcal{F}=(\mathcal{F}_k)_{k=0,\dots,K}$
  be the natural filtration of $\mathbf{E}=(E_1,\dots,E_K)$.

  Let $\tau$ be any ($\mathcal{F}$-)stopping time.
  First, we claim  that $\E[F_\tau(\mathbf{E})]=1$ holds.
  To show this claim, for $j=1,\dots,K-1$ define
  \[
    \tau_j
    :=
    \begin{cases}
      j & \text{if $\tau>j$} \\
      j+1 & \text{if $\tau\le j$}.
    \end{cases}
  \]
  Clearly, $\tau_j$ is a stopping time for each $j$. 
  Moreover, the realization of $(\tau,\tau_1,\dots,\tau_{K-1})$
  is always a permutation of $(1,\dots,K)$.
  Hence, using (c), we have 
  \[
    \E
    \left[
      F_\tau(\mathbf{E})+ \sum_{j=1}^{K-1}F_{\tau_j}(\mathbf{E})
    \right]
    =
    \E
    \left[
      \sum_{k=1}^{K }F_{k}(\mathbf{E})
    \right]
    =
    K.
  \]
  Using (b), $\E[F_{\tau_j}(\mathbf{E})]\le 1$ for each $j$.
  This implies $\E[F_{\tau }(\mathbf{E})]\ge 1$. 
  Therefore, $\E[F_\tau(\mathbf{E})]=1$ for any stopping time $\tau$.

  If, for some $k=0,\dots,K-1$,
  the event $A:=\{\E[F_{k+1}(\mathbf{E})\mid\FFF_{k}]>F_{k}(\mathbf{E})\}$
  has a positive probability, 
  then $\eta := k\id_{A}+K\id_{A^c}$ and $\eta':= (k+1)\id_{A}+K\id_{A^c}$,
  which are stopping times by (a),
  satisfy $\E[F_{\eta}(\mathbf{E})]<\E[F_{\eta'}(\mathbf{E})]$,
  violating the property that $\E[F_\tau(\mathbf{E})]=1$ for any stopping time $\tau$.
  Hence, $\P(A)=0$.
  Similarly, $\P(\E[F_{k+1}(\mathbf{E})\mid\FFF_{k}] < F_{k}(\mathbf{E})) = 0$.
  Therefore, $\E[F_{k+1}(\mathbf{E})\mid \FFF_{k}] = F_{k}(\mathbf{E})$ almost surely,
  and $(F_k (\mathbf{E}))_{k=0,\dots,K}$ is an $\FFF$-martingale.
 
  Note that $F_K$ is an se-merging function.
  By Theorem \ref{thm:inverse}, we have $F_K(\mathbf{E})\le S_K(\mathbf{E})$
  for some game martingale $(S_k)_{k=0,\dots,K}$.
  Since $\E[F_K(\mathbf{E})]=\E[S_K(\mathbf{E})]$,
  we have $F_K(\mathbf{E})= S_K(\mathbf{E})$ almost surely.
  Using the fact that both $(F_k(\mathbf{E}))_{k=1,\dots,K}$ and $(S_k(\mathbf{E}))_{k=1,\dots,K}$
  are martingales,
  we have, for $k=1,\dots,K$, almost surely
  \[
    F_{k}(\mathbf{E})
    =
    \E[F_{K}(\mathbf{E})\mid\FFF_{k}]
    =
    \E [S_K(\mathbf{E})\mid\FFF_{k}]
    =
    S_{k}(\mathbf{E}).
  \]
  Since $\mathbf{E}$ is arbitrary, we have $F_k=S_k$ for all $k$. 
\end{proof}

Theorem \ref{th:2} may be regarded as a counterpart for game martingales
of a classic result in probability theory that the optional stopping property characterizes martingality
(as in \cite[Theorem II.77.6]{Rogers/Williams:2000}).

Theorem \ref{th:2} implies that,
in order to get an anytime-valid and precise method for merging sequential e-values
(i.e., to get a ``precise e-process''),
the only tool one could rely on is the class of game martingales.
Another statement of this kind is established in \cite[Theorem 18]{RRLK20}:
the class of e-processes coincides
with the class of nonnegative processes dominated by test martingales.
Nevertheless, the setting in \cite{RRLK20} is different from ours;
for instance, the structure of e-processes for composite hypotheses
is more complicated than that for simple hypotheses,
whereas the results on se-merging functions in our paper
do not depend on whether we consider simple or composite hypotheses.

Combining Theorems \ref{thm:inverse} and \ref{th:2},
an admissible se-merging function must be the last component
of a  sequence of adapted, anytime valid, and precise functions.

\section{Merging independent e-values}
\label{sec:independent}

In this section, we will study the more delicate situation of merging independent e-values.
Since independent e-variables are sequential e-variables,
the martingale merging methods of Section \ref{sec:sequential} are valid in this situation.
An interesting question is whether many more merging functions are allowed for independent e-values. 

We first explain the intuition and then give formal definitions.
An important observation is that, in the case where the e-values to be merged are independent,
one may process them in any arbitrary order,
instead of the fixed order for sequential e-values.
Let us imagine that the $K$ e-values are written on cards
which initially are lying face down so that their values are not shown.
The statistician reveals these cards one by one
and bets on each card right before revealing it.
In the case of sequential e-values,
the order of revealing these cards is fixed (from card $1$ to card $K$),
and the relative bet on card $k+1$ (which has value $e_{k+1}$ on it)
is $s(\mathbf{e}_{(k)})$ in \eqref{eq:S}.

If the e-values are independent,
the statistician can decide the order of revealing these e-values.
Moreover, he can apply an adaptive strategy for turning over the cards,
that is, at each step, revealing a card based on what he has seen on the previous cards
(this picture goes back to Kolmogorov \cite[Section~2]{Kolmogorov:1963}).
As in the case of constructing martingale merging functions,
the statistician also needs to decide  the amount of bet at each step.
Therefore, the strategy involves two decision variables,
which e-value to reveal next, denoted by $\pi_{k+1}$,
and how much to bet on this e-value, denoted by $s(\mathbf{e}^\pi_{(k)})$,
where $\mathbf{e}^\pi_{(k)}:=(e_{\pi_1},\dots,e_{\pi_k})$
and again $\mathbf{e}^\pi_{(0)}:=\Box$.
Although not explicitly reflected in the notation,
one should keep in mind that $\pi_{k+1}$ is a function of $\mathbf{e}^\pi_{(k)}$.
We write $\mathbf{e}^\pi := (e_{\pi_1},\dots,e_{\pi_K})$,
which is the vector of e-values reordered by~$\pi$.

Certainly, an average of ie-merging functions obtained by this procedure
with different orders or different strategies
is still an ie-merging function.
This corresponds to a randomized betting strategy for the statistician:
at each step, he can generate $\pi_{k+1}$ and $s(\mathbf{e}^\pi_{(k)})$
using some distributions that depend on the past observations and decisions. 

We will make this procedure rigorous below.
The order of revealing the e-values is modelled by $\pi=(\pi_k)_{k\in[K]}$,
where $\pi_k:[0,\infty)^{k-1}\to[K]$ are measurable functions,
and for any $\mathbf{e}\in [0,\infty)^K$,
$\pi_k(\mathbf{e}^\pi_{(k-1)}) \ne \pi_j(\mathbf{e}^\pi_{(j-1)})$ for distinct $j,k\in[K]$,
meaning that you can only read the same e-value once.
Equivalently, $\{\pi_k(\mathbf{e}_{(k-1)}^\pi):k\in [K]\}=[K]$.
Such $\pi$ is called a \emph{reading strategy}.
 
For a given gambling system  $s:[0,\infty)^{<K}\to[0,1]$,
a reading strategy $\pi=(\pi_1,\dots,\pi_K)$,
and $c\in[0,1]$,
the corresponding \emph{reordered game martingale}
is the sequence of measurable functions $S^{s,\pi}_k:[0,\infty)^K\to[0,\infty)$,
$k=0,\dots,K$, given by $S_0^{s,\pi}:=c$ and
\begin{align}\label{eq:S-reordered}
  S^{s,\pi}_{k+1}(\mathbf{e})
  :=
  S^{s,\pi}_k(\mathbf{e})  
  \left(
    s(\mathbf{e}^\pi_{(k)}) e_{\pi_{k+1} }
    +
    1-s(\mathbf{e}^\pi_{(k)})
  \right), 
  \quad
  k=0,\dots,K-1,
\end{align}
where  we omit the argument in $\pi_{k+1} =\pi_{k+1}( \mathbf{e}^\pi_{(k)})$.
An equivalent way to write \eqref{eq:S-reordered} is 
\[
  S_{k}^{s,\pi}(\mathbf{e})
  =
  S_k (\mathbf{e}^\pi),
  \qquad
  k=0,1,\dots,K,
\]
where $(S_k)_{k=0,1,\dots,K}$ is the game martingale in \eqref{eq:S} associated with $s$.

A \emph{generalized martingale merging function} $F$
is a mixture (average) of $S_K^{s,\pi}$ above, that is, 
\begin{equation}\label{eq:F}
  F(\mathbf{e})
  =
  \E^{\mu} [S^{s,\pi} _K(\mathbf{e})]
  = 
  \int S^{s,\pi} _K(\mathbf{e})
  \mu(\d s,\d\pi)
  \text{\qquad for $\mathbf{e}\in[0,\infty)^K$},
\end{equation}
for some probability measure $\mu$ on the pairs $(s,\pi)$. 
Note that $s$ and $\pi$ are not independent in general.
 
The generalized martingale merging functions form a subclass of the ie-merging functions,
as the following theorem shows.

\begin{theorem}\label{th:3}
  Any generalized martingale merging function is an ie-merging function.
\end{theorem}

Theorem \ref{th:3} essentially follows from the following lemma.
For a proof of Theorem \ref{th:3}, we only need the first statement of the lemma.

\begin{lemma}\label{lem:reading-st}
  Let $E_1,\dots,E_K$ be independent e-variables, and $\pi$ be a reading strategy.
  Recursively define $E^\pi_k= E_{\pi_{k}(E_1^\pi,\dots,E_{k-1}^\pi)}$ for $k\in [K]$.
  Then $E_1^\pi,\dots,E_{K}^\pi$ are sequential e-variables.
  If $E_1,\dots,E_K$ are iid, then so are $E_1^\pi,\dots,E_{K}^\pi$.
\end{lemma}
\begin{proof}
  Set
  \begin{align*}
    \mathbf{E}^\pi_{(k)}
    &=
    (E_1^\pi,\dots,E_k^\pi)
    \text{ for $k\in\{0,1,\dots,K\}$}\\
    A^\pi_k
    &=
    [K]\setminus\{\pi_1(\mathbf{E}^\pi_{(0)}),\pi_2(\mathbf{E}^\pi_{(1)}),\dots,\pi_{k}(\mathbf{E}^\pi_{(k-1)})\}
    \text{ for $k\in [K]$};
  \end{align*}
  the latter is the set of all possible values that $\pi_{k+1}(\mathbf{E}^\pi_{(k)})$ can take.
  Let $f:\R\to\R$ be a Borel function.
  For $k=1,\dots,K-1$,
  \begin{align*}
    \E
    \left[
      f(E^\pi_{k+1}) \mid \mathbf{E}^\pi_{(k)}
    \right]
    &=
    \sum_{j\in[K]}
    \E
    \left[
      f(E_{j}) \id_{\{j=\pi_{k+1}(\mathbf{E}^\pi_{(k)})\}} \mid \mathbf{E}^\pi_{(k)}
    \right]
    \\
    &=
    \sum_{j\in  A^\pi_k} \id_{\{j=\pi_{k+1}(\mathbf{E}^\pi_{(k)})\}}
    \E
    \left[
      f(E_{j})\mid\mathbf{E}^\pi_{(k)}
    \right]. 
  \end{align*} 
  Using the independence between $E_j$ and $(E_k)_{k\ne j}$,
  for any fixed $j\in[K]$,
  \begin{align*}
    \E
    \left[
      f(E_{j})\mid  \mathbf{E}^\pi_{(k)}
    \right]
    =
    \id_{\{j \in A^\pi_k \}} \E\left[f(E_{j})\right]
    +
    \id_{\{j\notin A^\pi_k\}}
    f(E_j).
  \end{align*}
  Therefore,
  \begin{align}
    \E
    \left[
      f(E^\pi_{k+1}) \mid \mathbf{E}^\pi_{(k)}
    \right]
    =
    \sum_{j\in A^\pi_k} \id_{\{j=\pi_{k+1}(\mathbf{E}^\pi_{(k)})\}}
    \E
    \left[
      f(E_{j})
    \right].
    \label{eq:proof-lem3}
  \end{align}    
  Taking $f$ as the identity in \eqref{eq:proof-lem3}, we get 
  \[
    \E
    \left[
      E^\pi_{k+1} \mid \mathbf{E}^\pi_{(k)}
    \right]
    \le
    \sum_{j\in A^\pi_k} \id_{\{j=\pi_{k+1}(\mathbf{E}^\pi_{(k)})\}}
    =
    1.
  \] 
  This shows that  $E_1^\pi,\dots,E_{K}^\pi$ are sequential e-variables.  
 
  If $E_1,\dots,E_K$ are iid, then \eqref{eq:proof-lem3} yields
  $\E  [ f(E^\pi_{k+1}) \mid \mathbf{E}^\pi_{(k)}]=\E[f(E_1)]$
  and hence $\E[f(E^\pi_{k+1})]=\E[f(E_1)] $ for any Borel $f$.
  This shows that $E_{k+1}^\pi$ is independent of $E_1^\pi,\dots,E_k^\pi$ and identically distributed as $E_1$.
  Hence, $E_1^\pi,\dots,E_K^\pi$  are iid. 
\end{proof}

In the first statement of Lemma \ref{lem:reading-st},
$E^\pi_1,\dots,E^\pi_K$  are sequential e-variables but not necessarily independent.
For instance,
consider a setting where $E_1$ and $E_2$ are uniform on $[0,2]$ and $E_3=1$.
The strategy $\pi$ is specified by $\pi_1=1$, $\pi_2(e)= 2\id_{\{e<1\}} + 3\id_{\{e\ge 1\}}$.
It is clear that $E^\pi_1$ and $E^\pi_2$ are not independent.

\begin{proof}[Proof of Theorem \ref{th:3}]
  Let $\mathbf{E}=(E_1,\dots,E_K)$ be a vector of independent e-variables.
  It suffices to consider $S^{s,\pi}_K(\mathbf{E})$ in \eqref{eq:S-reordered}
  for a deterministic pair $(s,\pi)$,
  because a generalized martingale merging function is a mixture of $S_K^{s,\pi}$,
  and taking a mixture of e-variables yields an e-variable.
  Note that $S_K^{s,\pi}(\mathbf{E})=S_K(\mathbf{E}^\pi)$,
  where $\mathbf{E}^\pi=(E_1^\pi,\dots,E_K^\pi)$ in Lemma \ref{lem:reading-st}.
  Hence, $\E[S^{s,\pi}_K(\mathbf{E})]\le 1$ follows from Lemmas~\ref{lem:mart_is_se} and~\ref{lem:reading-st}. 
\end{proof}

The following is an example of a generalized martingale merging function
that is not an se-merging function.

\begin{example}
  The following function is taken from \cite[Remark 4.3]{VW21},
  where it is shown that
  \begin{equation}\label{eq:example}
    F(e_1,e_2)
    :=
    \frac12
    \left(
      \frac{e_1}{1 + e_1}
      +
      \frac{e_2}{1 + e_2}
    \right)
    \left(
      1 + e_1 e_2
    \right)
  \end{equation}
  is an ie-merging function.
  Let us check that $F$ is also a generalized martingale merging function.
  Notice that the symmetric expression \eqref{eq:example}
  can be represented as the arithmetic average of
  \begin{equation*}
    \frac{e_1}{1 + e_1}
    \left(
      1 + e_1 e_2
    \right)
    =
    e_1
    \left(
      \frac{1}{1 + e_1}
      +
      \frac{e_1}{1 + e_1}
      e_2
    \right)
  \end{equation*}
  and the analogous expression with $e_1$ and $e_2$ interchanged.
  The ``generalized gambling system'' producing \eqref{eq:example}
  starts from uncovering $e_1$ or $e_2$ with equal probabilities
  and investing all the capital in the chosen e-variable.
  If $e_1$ is uncovered first,
  it then invests a fraction of $e_1/(1+e_1)$ of its current capital into $e_2$.
  And if $e_2$ is uncovered first,
  it invests a fraction of $e_2/(1+e_2)$ of its current capital into $e_1$.

  Let us now check that $F$ is not an se-merging function.
  By the symmetry of $F$, we can assume, without loss of generality,
  that we first observe the e-variable $E_1$ producing $e_1$
  and then observe $E_2$ producing $e_2$.
  Had $F$ been an se-merging function,
  \[
    G(e_1)
    :=
    \sup_{E_2\in\EEE}\E [F(e_1,E_2)]
  \]
  would have produced an e-variable when plugging in $e_1=E_1$.
  Let $E_2$ be given by $p^{-1}\id_{A}$ where $\P(A)=p$ for some $p\in (0,1]$; 
  later $p$ will be chosen depending on $e_1$.
  We can compute
  \begin{align*}
    \E[F(e_1,E_2)]
    &=
    (1-p)\frac{1}{2}
    \frac{e_1}{1+e_1}
    +
    p\frac{1}{2}
    \left(\frac{e_1}{1+e_1} + \frac{p^{-1}}{1+p^{-1}} \right)
    \left (1+e_1p^{-1}\right) \\
    &=
    \frac{1}{2}
    \left(
      e_1 + \frac{p+e_1}{p+1}
    \right).
  \end{align*}
  Letting $p\downarrow 0$ we have $\E [F(e_1,E_2)]\to e_1$,
  and letting $p:=1$ we have $\E[F(e_1,E_2)] = (3e_1+1)/4$.
  Therefore,
  \begin{equation*}
    G(e_1)
    =
    \sup_{E_2\in\EEE}\E[ F(e_1,E_2)]
    \ge
    \max
    \left(
      e_1,\frac{3e_1+1}{4}
    \right),
  \end{equation*} 
  Note that if $\E[E_1]=1$, then $\E[ G(E_1)]>1$
  unless $E_1=1$ with probability $1$.
  Hence, $F$ is not an se-merging function.
\end{example}

\begin{example}
  In the case $K=2$, the reading strategy $\pi$ is specified by $\pi_1$,
  which does not depend on any observed e-values.
  Hence, the bet $s$ can be chosen separately
  on the events $\{\pi_1=1\}$ and $\{\pi_1=2\}$.
  Therefore, we can write all generalized martingale merging functions in \eqref{eq:F} as
  \begin{equation}\label{eq:decompose}
    (e_1,e_2)
    \mapsto
    \beta (1 + a_1 \tilde e_1) (1 + g_1 (e_1) \tilde e_2)
    +
    (1-\beta) (1+ a_2 \tilde e_2) (1+ g_2 (e_2) \tilde e_1), 
  \end{equation}
  where $\tilde e_1=e_1-1$, $\tilde e_2=e_2-1$, $\beta,a_1,a_2\in [0,1]$,
  and $g_1,g_2:[0,\infty)\to[0,1]$ are Borel functions.
  To interpret \eqref{eq:decompose} as a randomized reordered game martingale,
  $\beta $ is the probability of $\pi_1=1$,
  $a_1,a_2$ are the first-round bets, and $g_1,g_2 $ are the second-round bets.
\end{example}

In view of Theorem \ref{thm:inverse},
one may wonder whether any ie-merging function is dominated
by a generalized martingale merging function.
The answer is, somewhat surprisingly, no,
as the following counter-example shows.\footnote%
  {This counter-example is due to Zhenyuan Zhang, to whom we are grateful.\label{f:ZZ}}
\begin{example}\label{ex:ct}
  Fix a constant $c>1$ and define the function $G:[0,\infty)^2\to\R$ by
  \[
    G(\mathbf{e})=\id_{[0,c)^2 }(\mathbf{e})+(2c-1) \id_{[c,\infty)^2}(\mathbf{e}).
  \]
  Take any two independent e-variables $E_1,E_2$,
  and set $a:=\P(E_1\ge c)$ and $b:=\P(E_2\ge c)$.
  Markov's inequality gives $a,b\le1/c$,
  which implies $1/a+1/b\ge2c$, and thus $a+b\ge2cab$.
  We can verify
  \begin{align*}
    \E[G(E_1,E_2)]&=\P(E_1 < c)\P(E_2< c)+(2c-1)\P(E_1\ge c)\P(E_2\ge c)\\
    &=1-a-b+2cab\le 1.
  \end{align*} 
  Hence, $G$ is an ie-merging function.  
  Let us show that $G$ is not dominated by any generalized martingale merging function $F$.
  Suppose otherwise.
  We can assume that $F$ has the form \eqref{eq:decompose}.
  Since $F \ge G$, we know that $F=1$ on $[0,1]^2$,
  which implies that $a_1=a_2=0$
  (unless $\beta\in\{0,1\}$,
  in which case assuming $a_1=a_2=0$ does not lead to loss of generality).
  This gives $F(c,c) \le \beta c+(1-\beta)c = c$. 
  However, $G(c,c)=2c-1$ whereas $F(c,c)\le c<2c-1$, a contradiction.
\end{example}

\section{Conclusion}
\label{sec:conclusion}

For sequential e-variables,
full characterizations of admissible merging functions and methods for constructing e-processes
are obtained in this paper.
For independent e-variables,
we propose the class of generalized martingale merging functions,
but a full picture remains unclear.

Regarding the merging functions for independent or sequential e-values,
there are several open questions.
\begin{enumerate}
\item
  It is unclear in which practical settings
  constructing independent e-values performs better than constructing sequential but dependent e-values. 
  Since independent e-values are more restrictive to construct
  (cf.\ the example of likelihood ratios in Section \ref{sec:example}), 
  it would only be valuable to construct them
  if they carry more statistical power in some situations. 
\item
  It would be interesting to find simple, and perhaps practically useful,
  ie-merging functions that are not dominated by a generalized martingale merging function.
\item 
  It remains unclear whether any increasing ie-merging function
  is dominated by a generalized martingale merging function;
  note that the function $G$ in Example \ref{ex:ct} is not increasing. 
\item
  We may further require an ie-merging function to be precise;
  that is, for independent e-variables with mean $1$, the function produces an e-variable with mean $1$.
  Is every ie-merging function with this requirement necessarily a generalized martingale merging function?
\item
  Instead of sequential e-variables, 
  we can consider a general filtration $\mathcal F = (\mathcal F_k)_{k\in\{0,\dots,K\}}$
  and a vector $\mathbf{E}=(E_k)_{k\in [K]}$ of adapted e-variables
  satisfying $\E[E_k|\mathcal F_{k-1}]\le 1$ for each $k$.
  By using the tower property,  $E_1,\dots,E_K$ are sequential e-variables. 
  Therefore,  all se-merging functions are valid merging functions
  for all such pairs $(\mathbf E,\mathcal F)$. 
  It may be possible to obtain a larger class of merging functions
  that are valid for all pairs $(\mathbf E,\mathcal F)$ satisfying some further conditions;
  the class of se-merging functions corresponds to the setting
  where $\mathcal F$ is the natural filtration of $E_1,\dots,E_K$.
  (If we relax the assumption that $\mathbf{E}$ is adapted to $\mathcal F$,
  then by considering all pairs $(\mathbf E,\mathcal F)$, 
  we arrive at the class of all e-merging functions.)
\end{enumerate}

\subsection*{Acknowledgments}

We are grateful to the anonymous referees of the journal version of this paper
(to appear in the \emph{Electronic Journal of Statistics})
for helpful comments.
For the role of Zhenyuan Zhang, see footnote~\ref{f:ZZ}.

The first author was supported by Amazon, Stena Line, and Mitie.
The second author was supported in part by grants CRC-2022-00141 and RGPIN-2018-03823
from the Natural Sciences and Engineering Research Council of Canada.

\end{document}